\documentclass[11pt,a4paper]{amsart}
\usepackage[latin9]{inputenc}
\usepackage[paperwidth=8.5in,paperheight=11in]{geometry}
\usepackage{refstyle}
\usepackage{textcomp}
\usepackage{amstext}
\usepackage{amsthm}
\usepackage{amssymb}
\usepackage{graphicx}
\usepackage[unicode=true,pdfusetitle,
 bookmarks=true,bookmarksnumbered=false,bookmarksopen=false,
 breaklinks=false,pdfborder={0 0 1},backref=false,colorlinks=false]
 {hyperref}

\numberwithin{equation}{section}
\numberwithin{figure}{section}
  \theoremstyle{definition}
  \newtheorem{defn}{\protect\definitionname}[section]
  \theoremstyle{remark}
  \newtheorem*{rem*}{\protect\remarkname}
  \theoremstyle{plain}
  \newtheorem{prop}{\protect\propositionname}[section]
  \theoremstyle{plain}
  \newtheorem{thm}{\protect\theoremname}[section]
  \theoremstyle{plain}
  \newtheorem*{thm*}{\protect\theoremname}
  \theoremstyle{plain}
  \newtheorem{cor}{\protect\corollaryname}[section]
  \theoremstyle{definition}

\usepackage{mathtools}
\usepackage{tikz-cd}

\newref{def}{refcmd={Definition \ref{#1}}}
\newref{fig}{refcmd={Figure \ref{#1}}}
\newref{thm}{refcmd={Theorem \ref{#1}}}
\newref{prop}{refcmd={Proposition \ref{#1}}}

\makeatother

  \providecommand{\definitionname}{Definition}
  \providecommand{\examplename}{Example}
  \providecommand{\propositionname}{Proposition}
  \providecommand{\remarkname}{Remark}
\providecommand{\corollaryname}{Corollary}
\providecommand{\theoremname}{Theorem}

\begin{document}
\global\long\def\Tor{\operatorname{Tor}}
\global\long\def\gr{\operatorname{gr}}
\global\long\def\tot{\operatorname{tot}}
\global\long\def\B{\mathrm{B}}

\title{The Bracket in the Bar Spectral Sequence for an Iterated Loop Space}
\author{Xianglong Ni}
\maketitle

\thispagestyle{empty}
\begin{abstract}
When $X$ is an associative H-space, the bar spectral sequence computes
the homology of the delooping, $H_{*}(BX)$. If $X$ is an $n$-fold
loop space for $n\geq2$ this is a spectral sequence of Hopf algebras.
Using machinery by Sugawara and Clark, we show that the spectral sequence
filtration respects the Browder bracket structure on $H_{*}(BX)$,
and so it is moreover a spectral sequence of Poisson algebras. Through
the bracket on the spectral sequence, we establish a connection between
the degree $n-1$ bracket on $H_{*}(X)$ and the degree $n-2$ bracket
on $H_{*}(BX)$. This generalizes a result of Browder and puts it
in a computational context.
\end{abstract}

\section{\label{sec:Introduction}Introduction}

For an associative H-space $G$, the Milnor-Dold-Lashof construction
builds the delooping $BG$. The algebraic analogue is the bar construction,
and this relationship gives rise to the bar spectral sequence (also
known as the Rothenberg-Steenrod spectral sequence) which computes $H_{*}(BG)$
from $H_{*}(G)$. The case where $G$ is an infinite
loop space has been extensively studied: working at a prime $p$, Ligaard and Madsen
\cite{DL in bar ss} showed that there is a Dyer-Lashof action on
the spectral sequence compatible with the ones on $H_{*}(G)$ and
$H_{*}(BG)$.


Here, we will instead consider spaces $\Omega^n X$ where $n$ is finite, and we will work over an arbitrary field.
$S^{n-1}$ parametrizes the multiplications on $\Omega^n X$, producing Browder's bracket in homology---this process is reviewed in \S\ref{subsec:Browder-bracket}. This bracket shifts degree by $n-1$ and is an obstruction to $\Omega^n X$ being an $(n+1)$-fold loop space. Likewise, $H_*(\Omega^{n-1} X)$ has a bracket of degree $n-2$.

We show that the bar construction of $H_*(\Omega^n X)$ receives a bracket as well, and then prove the following theorem:

\begin{thm*}
	The bar spectral sequence
	\[
		E^2_{*,*}\cong\Tor_{*,*}^{H_*(\Omega^n X)}(k,k)\Rightarrow H_*(\Omega^{n-1}X)
	\]
	is a spectral sequence of Poisson Hopf algebras. The bracket in the spectral sequence shifts bidegree by $(-1,n-1)$ and satisfies
	\[
		d^r[x,y] = [d^rx,y] + (-1)^{n+|x|}[x,d^r y].
	\]
	Moreover, the bracket on the $E^1_{*,*}$ page is the bracket in the bar construction of $H_*(\Omega^n X)$. 
\end{thm*}
The vertical shift of $n-1$ comes from the bracket on $H_*(\Omega^n X)$. Together with the horizontal shift of $-1$, the bracket has total degree $n-2$, matching the one on $H_*(\Omega^{n-1}X)$.

In \S\ref{sec:Background} we review the relevant definitions and constructions.
Then, in \S\ref{sec:combinatorics-bracket-bar-construction} we give a formula for the bracket on the bar construction and show that it is well-behaved. In \S\ref{sec:bracket-mainproof-section} we set up and prove the theorem stated above,
and we show that the bracket on $E^1_{*,*}$ of the spectral sequence is the one on $\B_{*,*}(H_*(\Omega^n X))$ as described in \S\ref{sec:combinatorics-bracket-bar-construction}. We obtain
Theorem 2-1 of \cite{browder bracket} as a special case. For technical reasons,
we consider $n=2$ separately from $n\geq3$.

This machinery can be used to infer properties of the bracket on $H_*(\Omega^n X)$ from the one on $H_*(\Omega^{n-1}X)$. In some cases it determines it completely; $H_*(\Omega^n S^k; \mathbb{Q})$ is such an example, and the reader is invited to compute it. There also appears to be related phenomena in the topological Hochschild homology of $E_n$ ring spectra.

This paper was written as part of the 2017 Summer Program in Undergraduate
Research (SPUR) at MIT directed by Slava Gerovitch. The author would
like to thank David Jerison and Ankur Moitra for their guidance, his
mentor Robert Hood Chatham for many long insightful discussions, and his
advisor Haynes Miller for providing this intriguing project and giving detailed feedback on this paper.

\section{\label{sec:Background}Background}

Given a space $X$, the loop space $\Omega X$ is typically modeled
as the space of all paths $I\rightarrow X$ sending $\{0,1\}$ to $*\in X$,
the basepoint. This model has a defect: the multiplication (given
by concatenation and doubling the speed) is only unital and associative
up to homotopy. To remedy this, we instead adopt the following homotopy
equivalent model, called the ``Moore loops'' of $X$ after J. C. Moore \cite{Moore56}.
\begin{defn}[Moore Loops]\label{def:Moore-loops}
Given a space $X$ with basepoint $*$, let $X^\mathbb{R}$ denote the space of all continuous maps from $\mathbb{R}$ to $X$ which send $0$ to $*$. Define $\Omega X$ to be the subspace
of $X^{\mathbb{R}}\times\mathbb{R}_{\geq0}$ consisting of pairs $(\alpha,l)$
such that $\alpha(t)=*$ if $t\leq0$ or $t\geq l$. That is, $\Omega X$
consists of loops in $X$ together with their lengths. $\Omega X$
is itself a pointed space, with basepoint the constant path $c_{*}$
at $*\in X$ of length 0, and moreover it is an H-space with the multiplication
$(\alpha,l_{\alpha})(\beta,l_{\beta})=(\omega,l_{\alpha}+l_{\beta})$
where
\[
\omega(t)=\begin{cases}
\alpha(t) & \text{if }t\leq l_{\alpha},\text{ and}\\
\beta(t-l_{\alpha}) & \text{if }l_{\alpha}\leq t.
\end{cases}
\]
\end{defn}
The benefit of this model is that $\Omega X$ is a strictly associative
H-space with unit $c_{*}$. Let $C_{*}$ denote
the normalized singular chain complex, which is the singular chain
complex modulo degenerate chains. The complex $C_{*}(\Omega X)$
is a differential graded algebra (henceforth abbreviated DGA) with
the multiplication
\[
C_{*}(\Omega X)\otimes C_{*}(\Omega X)\xrightarrow{\text{EZ}}C_{*}(\Omega X\times\Omega X)\rightarrow C_{*}(\Omega X)
\]
where the first map is the Eilenberg-Zilber map and the latter is
induced by the multiplication on $\Omega X$. In turn, this induces
a multiplication on $H_{*}(\Omega X)$---the Pontryagin product.
\begin{rem*}[Notation]
Throughout, we will consider only connected spaces.
Coefficients are taken in a field $k$ to ensure that the K\"{u}nneth
map is an isomorphism. As such, this condition can be relaxed to $k$
a commutative ring and $H_{*}(\Omega^n X;k)$ flat.

If $x$ is an element of a (bi)graded object, then we take $\left|x\right|$
to mean its total degree. ``Commutative'' will mean ``commutative
in the graded sense'' so that $xy=(-1)^{\left|x\right|\left|y\right|}yx$.
\end{rem*}

\subsection{The bar construction\label{subsec:The-bar-construction}}

The following is a brief summary of the bar construction. For more details than what is provided here, see \cite{Eilenberg-MacLane53} and \cite{sseq user's guide}. We begin
in a purely algebraic context and then specialize to the case of topological
interest at the end of \S\ref{subsec:bar-sseq}. Consider a
DGA $(\mathcal{A},d_{\mathcal{A}})$ with augmentation $\epsilon\colon\mathcal{A}\rightarrow k$.
\begin{defn}
\label{def:bar-construction}Let $\overline{\mathcal{A}}=\ker\epsilon$.
The (normalized) \emph{bar construction} $\B_{*,*}(\mathcal{A})$
is a bigraded $k$-module with
\[
\B_{s,*}(\mathcal{A})=\underset{s\text{ times}}{\underbrace{\overline{\mathcal{A}}\otimes\cdots\otimes\overline{\mathcal{A}}}}
\]
of which the component in degree $t$ is denoted $\B_{s,t}(\mathcal{A})$.
We call $s$ the \emph{external degree}, and $t$ the \emph{internal
degree}, and $s+t$ the \emph{(total) degree}.\emph{ }It is conventional
to write $\alpha_{1}\otimes\cdots\otimes\alpha_{s}\in\B_{s,*}(\mathcal{A})$
as $[\alpha_{1}|\cdots|\alpha_{s}]$. The bar construction is a double
complex with an \emph{internal differential} $d$ of bidegree $(0,-1)$
and an \emph{external differential} $\delta$ of bidegree $(-1,0)$,
defined as
\begin{align*}
d[\alpha_{1}|\cdots|\alpha_{s}] & =\sum_{i=1}^{s}(-1)^{\sigma(i-1)}[\alpha_{1}|\cdots|d_{\mathcal{A}}\alpha_{i}|\cdots|\alpha_{s}],\\
\delta[\alpha_{1}|\cdots|\alpha_{s}] & =\sum_{i=1}^{s-1}(-1)^{\sigma(i)}[\alpha_{1}|\cdots|\alpha_{i}\alpha_{i+1}|\cdots|\alpha_{s}],
\end{align*}
where the sign is given by $\sigma(i)=\left|[\alpha_{1}|\cdots|\alpha_{i}]\right|$.
We will often use $d$ to denote both the differential on $\mathcal{A}$
and the internal differential in $\B_{*,*}(\mathcal{A})$, as there
is no risk of confusion.
\end{defn}
The homology of the total complex with differential $D=d+\delta$
computes $\Tor_{*}^{\mathcal{A}}(k,k)$:
\begin{equation}
H_{*}(\tot\B_{*,*}(\mathcal{A}))\cong\Tor_{*}^{\mathcal{A}}(k,k).\label{eq:TorA(k,k)}
\end{equation}

The bar construction has a natural comultiplication $\Delta\colon\B_{*,*}(\mathcal{A})\rightarrow\B_{*,*}(\mathcal{A})\otimes\B_{*,*}(\mathcal{A})$
sending
\[
\Delta[\alpha_{1}|\cdots|\alpha_{s}]=\sum_{i=0}^{s}[\alpha_{1}|\cdots|\alpha_{i}]\otimes[\alpha_{i+1}|\cdots|\alpha_{s}]
\]
which gives $\Tor_{*}^{\mathcal{A}}(k,k)$ a coalgebra structure.
\begin{defn}
\noindent \label{def:shuffle-product}Let $p$ and $q$ be non-negative integers. A $(p,q)$\emph{-shuffle} $\varphi$ is a permutation of $\{1,\ldots,p+q\}$ such that $\varphi(a)<\varphi(b)$
if $1\leq a<b\leq p$ or if $p+1\leq a<b\leq p+q$. If $\mathcal{A}$
and $\mathcal{A}'$ are DGAs, the \emph{shuffle product} maps $\B_{*,*}(\mathcal{A})\otimes\B_{*,*}(\mathcal{A}')\rightarrow\B_{*,*}(\mathcal{A}\otimes\mathcal{A}')$,
and is defined as
\[
[\alpha_{1}|\cdots|\alpha_{p}]\otimes[\alpha_{1}'|\cdots|\alpha_{q}']\mapsto\sum_{(p,q)\text{-shuffles }\varphi}(-1)^{\sigma(\varphi)}[a_{\varphi^{-1}(1)}|\cdots|a_{\varphi^{-1}(p+q)}],
\]
where
\[
a_{i}=\begin{cases}
\alpha_{i}\otimes1 & \text{if }i\leq p,\\
1\otimes\alpha_{i-p}' & \text{if }i>q,
\end{cases}
\]
\[
\sigma(\varphi)=\sum_{\varphi(i)>\varphi(j+p)}(|\alpha_{i}|+1)(|\alpha_{j}'|+1).
\]
The shuffle product was introduced by Eilenberg and Mac Lane in \cite{Eilenberg-MacLane53}.

Each shuffle in the sum can be thought of as a walk from $(0,0)$ to $(p,q)$ that
goes rightwards or upwards at each step---see \figref{shuffle-product}. The identity shuffle is the
walk that goes through $(p,0)$, and deviation from this shuffle incurs
the usual signs from moving elements past each other. Note that $[\alpha_{i}]$
has bidegree $(1,|\alpha_{i}|)$ and thus total degree $|\alpha_{i}|+1$.
\end{defn}
\begin{figure}
	\includegraphics[width=6cm]{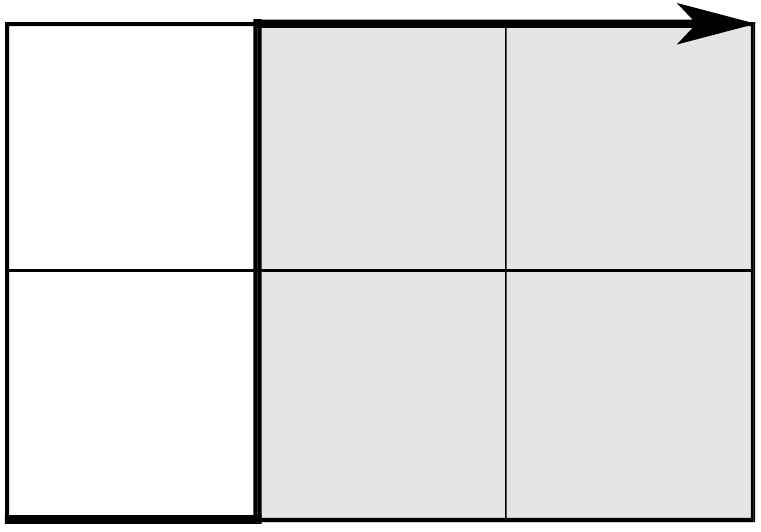}
	
	\caption{\label{fig:shuffle-product}A pictorial representation of the term $\pm [x_1 \otimes 1|1\otimes y_1 |1\otimes y_2|x_2 \otimes 1|x_3 \otimes 1]$ in the shuffle of $[x_1|x_2|x_3]$ and $[y_1|y_2]$. The sign is determined by the blocks under the walk.}
\end{figure}
The shuffle product gives $\B_{*,*}(\mathcal{A})$ a commutative product if $\mathcal{A}$ itself is commutative: the multiplication $\mathcal{A}\otimes\mathcal{A}\rightarrow\mathcal{A}$ is an algebra map and induces a map $\B_{*,*}(\mathcal{A}\otimes\mathcal{A})\rightarrow\B_{*,*}(\mathcal{A})$. Then the multiplication on $\B_{*,*}(\mathcal{A})$ is given by the composite
\[
\B_{*,*}(\mathcal{A})\otimes\B_{*,*}(\mathcal{A})\xrightarrow{\text{EZ}}\B_{*,*}(\mathcal{A}\otimes\mathcal{A})\rightarrow\B_{*,*}(\mathcal{A})
\]
where the first map is the shuffle product.

\subsection{\label{subsec:bar-sseq}The bar spectral sequence}

Since the preceding gives $\Tor_{*}^{\mathcal{A}}(k,k)$ as the total
homology of a first-quadrant double complex, we receive a strongly
convergent spectral sequence by filtering the total complex by external
degree $s$:
\[
F_{s}\left(\tot\B_{*,*}(\mathcal{A})\right)_{n}=\bigoplus_{\substack{p+q=n\\
p\leq s
}
}\B_{p,q}(\mathcal{A}).
\]
The associated graded of this filtration is
\begin{align*}
E_{s,t}^{0}=\gr_{s}\left(\tot\B_{*,*}(\mathcal{A})\right)_{s+t} & =\B_{s,t}(\mathcal{A})
\end{align*}
and the differentials on the first two pages are given by $d_{0}=d$,
$d_{1}=\delta$ as in \defref{bar-construction}. By the K\"{u}nneth formula,
$H_{*}(\B_{s,*}(\mathcal{A}))\cong\B_{s,*}(H_{*}(\mathcal{A}))$ and
the $E_{*,*}^{1}$ page is the bar complex on $H_{*}(\mathcal{A})$
if we treat it as a DGA with trivial differential. Therefore our spectral
sequence has the form
\begin{equation}
E_{*,*}^{2}\cong\Tor_{*,*}^{H_{*}(\mathcal{A})}(k,k)\Rightarrow\Tor_{*}^{\mathcal{A}}(k,k).\label{eq:algebraic-bar-ss}
\end{equation}

The topological significance of this construction stems from the fact
that, if $G$ is an associative H-space, there is a chain equivalence
\begin{equation}
\tot\B_{*,*}(C_{*}(G))\xrightarrow{\simeq}C_{*}(BG)\label{eq:topochainequiv}
\end{equation}
inducing the homology isomorphism
\begin{equation}
\Tor_{*}^{C_{*}(G)}(k,k)\cong H_{*}(BG;k).\label{eq:topologyiso}
\end{equation}
In general, (\ref{eq:algebraic-bar-ss}) is a spectral sequence of coalgebras and (\ref{eq:topologyiso}) is an isomorphism of coalgebras \cite{cartan-moore seminar}. Even if $\mathcal{A}$
itself is not strictly commutative, it may still be possible to grant
the total complex $\tot\B_{*,*}(\mathcal{A})$ a multiplication\textemdash this
is the case when $\mathcal{A}=C_{*}(\Omega G)$. We will return to this point in \S\ref{sec:bracket-2-fold}.
In this case, (\ref{eq:algebraic-bar-ss}) is a spectral sequence of Hopf algebras and (\ref{eq:topologyiso}) is an isomorphism of Hopf algebras \cite{homocomm}.

\subsection{\label{subsec:Browder-bracket}Browder's bracket}

Browder \cite{browder bracket} inductively defines\footnote{We are using the Moore loops model of $\Omega X$ for its benefit of strict associativity. If one instead uses the non-associative model where elements of $\Omega X$ are maps $(I,\partial I)\rightarrow(X,*)$, the
little cubes operad gives a map
\[
\tilde{\phi}\colon\mathcal{C}_{n}(2)\times\Omega^{n}X\times\Omega^{n}X\rightarrow\Omega^{n}X
\]
where $\mathcal{C}_{n}(2)\simeq S^{n-1}$.} a map $\phi\colon \Omega^{n}X\times S^{n-1} \times\Omega^{n}X\rightarrow\Omega^{n}X$
where $S^{n-1}$ parametrizes the choice of multiplication on $\Omega^{n}X$.

The map ${\phi}_{1}\colon \Omega X\times S^{0}\times \Omega X\rightarrow\Omega X$
sends ${\phi}(a,1,b)=ab$ and $\tilde{\phi}(a,-1,b)=ba$, where
multiplication is concatenation of loops as in \defref{Moore-loops}.

Given ${\phi}_{n}\colon \Omega^n X \times S^{n-1} \times \Omega^n X \to \Omega^n X$,
define a map $\Omega^{n+1} X\times S^{n-1}\times[-1,1]\times \Omega^{n+1} X\rightarrow\Omega^{n+1} X$
by
\[
(a,u,t,b)\mapsto\begin{cases}
\tilde{\phi}_{n}(c(l(b)t)a,u,b) & \text{if }t\geq0,\\
\tilde{\phi}_{n}(a,u,c(l(a)\left|t\right|)b) & \text{if }t\leq0,
\end{cases}
\]
where $l(a)$ denotes the length of $a$ and $c(t)$ is a constant
loop of length $t$ at the basepoint of $\Omega^n X$. This map factors through
\[
\Omega^{n+1} X \times S^n \times \Omega^{n+1} X \cong \Omega^{n+1} X \times \frac{S^{n-1}\times[-1,1]}{(S^{n-1}\times\{-1\})\cup (S^{n-1}\times\{1\})} \times \Omega^{n+1} X \rightarrow \Omega^{n+1} X
\]
which we take to be ${\phi}_{n+1}$. We will suppress the subscript since it can be inferred from context, and we will write $\phi_{*}$ both for the map induced on chains and for
the map induced on homology.
\begin{defn}
Let $\gamma\in H_{n-1}(S^{n-1})$ be a generator. Then $H_{*}(\Omega^{n}X)$
has a \emph{bracket}\footnote{Unfortunately, there are a multitude of square brackets appearing
in this work\textemdash and all of them are conventional notation.}
\[
[-,-]\colon H_{p}(\Omega^{n}X)\otimes H_{q}(\Omega^{n}X)\rightarrow H_{p+q+n-1}(\Omega^{n}X),
\]
\[
[x,y]=\phi_{*}(x\otimes\gamma\otimes y).
\]
\end{defn}

We remark that this bracket is related to Browder's $\psi$ by $[x,y]=(-1)^{(n-1)\left|x\right|}\psi(x,y)$. The sign difference is because he writes $S^{n-1} \times \Omega^n X \times \Omega^n X$ instead of $\Omega^n X \times S^{n-1} \times \Omega^n X$.

In \cite{Homology-Iterated-Loop-Spaces}, Cohen showed that $H_{*}(\Omega^{n}X)$ is a \emph{Poisson Hopf algebra} with this bracket of degree $n-1$. That is to say, the bracket is:
\begin{itemize}
	\item antisymmetric: $$[x,y]=-(-1)^{(|x|+n-1)(|y|+n-1)}[y,x],$$
	\item (Poisson identity) a derivation with respect to the multiplication: $$[x,yz]=[x,y]z+(-1)^{|y|(|x|+n-1)}y[x,z],$$
	\item (Jacobi identity) a derivation with respect to itself: $$[x,[y,z]]=[[x,y],z]+(-1)^{(|x|+n-1)(|y|+n-1)}[y,[x,z]],$$
	\item and compatible with the coproduct:
	\begin{equation}
		\Delta([x,y]) = [\Delta(x),\Delta(y)].\label{eq:coproduct-bracket}
	\end{equation}
	Here $[\Delta(x),\Delta(y)]$ means the bracket in the tensor product, which is given by the following formula (extended linearly):
	\begin{equation}\label{eq:Poisson-tensor-product}
	[x_1 \otimes x_2, y_1 \otimes y_2] = (-1)^{|x_2|(|y_1|+n-1)}[x_1, y_1] \otimes x_2 y_2 + (-1)^{(|x_2|+n-1)|y_1|} y_1 x_1 \otimes [x_2, y_2].
	\end{equation}
\end{itemize}

\section{\label{sec:combinatorics-bracket-bar-construction}The bracket in the bar construction}
If $\mathcal{A}$ is a commutative DGA, the bar construction $\B_{*,*}(\mathcal{A})$ is a commutative Hopf algebra (c.f. \S\ref{subsec:The-bar-construction}). Now we will show that if $\mathcal{A}$ has a bracket, then $\B_{*,*}(\mathcal{A})$ is endowed with one as well.

A \emph{differential graded Poisson algebra} is a DGA with an antisymmetric, bilinear bracket that satisfies the Poisson and Jacobi identities and moreover satisfies the Leibniz rule with respect to the differential. For example, we can consider $H_*(\Omega^n X)$ as a differential graded Poisson (Hopf) algebra, with trivial differential.
\begin{thm}\label{thm:poisson-hopf-alg-bar-construction}
	Let $\mathcal{A}$ be a commutative differential graded Poisson algebra with bracket of degree $n-1$. The bar construction $\B_{*,*}(\mathcal{A})$ is then a commutative differential bigraded Poisson Hopf algebra, where the bracket has bidegree $(-1,n-1)$ and satisfies Liebniz with respect to both the internal differential $d$ and the external differential $\delta$.
	
	Explicitly, if $x_1,\ldots,x_p,y_1,\ldots,y_q \in \mathcal{A}$, the bracket $\left[[x_{1}|\cdots|x_{p}],[y_{1}|\cdots|y_{q}]\right]$
	is given by (using the notation of \defref{shuffle-product})
	\begin{align}\label{eq:formula-for-bar-bracket}
	\sum_{(p,q)\text{-shuffles }\varphi}\sum_{\substack{\varphi^{-1}(i)\leq p\\
			\varphi^{-1}(i+1)>p
		}
	}(-1)^{\sigma(\varphi)}\big[a_{\varphi^{-1}(1)}|\cdots|(-1)^{\sigma'(\varphi,i)}[a_{\varphi^{-1}(i)},a_{\varphi^{-1}(i+1)}]|\cdots|a_{\varphi^{-1}(p+q)}\big]
	\end{align}
	\[
	a_{i}=\begin{cases}
	x_{i} & \text{if }i\leq p,\\
	y_{i-p} & \text{if }i>p,
	\end{cases}
	\]
	where the sign $(-1)^{\sigma(\varphi)}$ is from the shuffle product and ${\sigma'(\varphi,i)}$ is defined as
	\[
	\sigma'(\varphi,i)=n\left(\sum_{\varphi(j)>i+1}(\left|x_{j}\right|+1)+\sum_{\varphi(j)<i}(\left|y_{j}\right|+1)\right).
	\]
\end{thm}
If one thinks of a shuffle as a path traveling up and right from $(0,0)$
to $(p,q)$, then the terms in (\ref{eq:formula-for-bar-bracket}) are such paths with a single
bracket inserted on a lower-right corner. See \figref{shuffle-product-bracket}.
\begin{figure}
	\includegraphics[width=6cm]{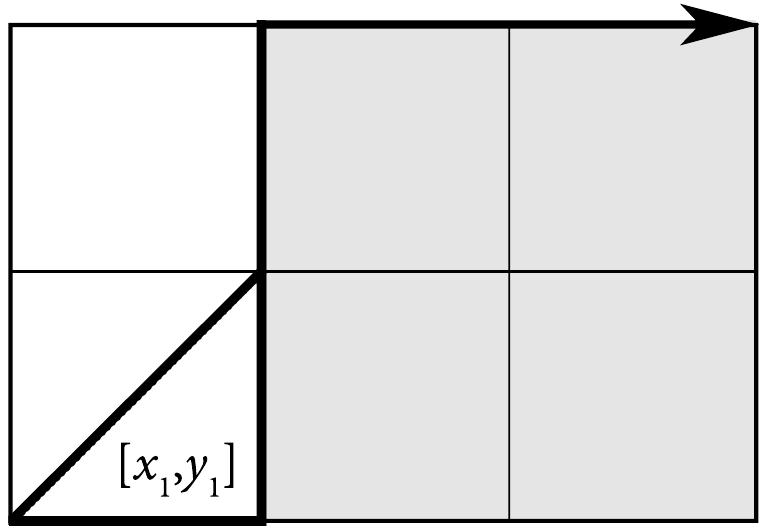}
	
	\caption{\label{fig:shuffle-product-bracket}The term $\pm \big[[x_1,y_1] |y_2|x_2|x_3\big]$ in the bracket $\big[[x_1|x_2|x_3],[y_1|y_2]\big]$.}
\end{figure}
\begin{proof}
	\begin{figure}[h]
		\includegraphics[width=6cm]{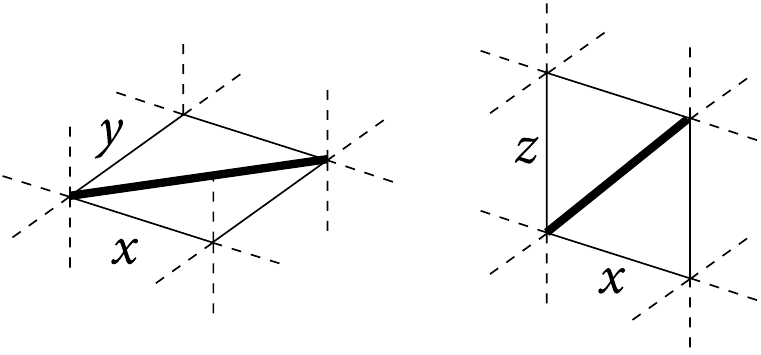}
		
		\caption{\label{fig:diagonals-for-Poisson-id}Brackets appearing in terms in (\ref{eq:Poisson-id-n-even}).}
	\end{figure}
	The needed axioms can be checked combinatorially. We check the Poisson identity as an example, in the case $n$ is even. Since the bracket has even total degree $n-2$, the Poisson identity states
	\begin{equation}
	[\mathbf{x},\mathbf{y}\mathbf{z}]-[\mathbf{x},\mathbf{y}]\mathbf{z}-(-1)^{|\mathbf{y}||\mathbf{x}|}\mathbf{y}[\mathbf{x},\mathbf{z}]=0 \label{eq:Poisson-id-n-even}
	\end{equation}
	where
	\begin{align*}
	\mathbf{x}&=[x_1|\cdots|x_p]\\
	\mathbf{y}&=[y_1|\cdots|y_q]\\
	\mathbf{z}&=[z_1|\cdots|z_r].
	\end{align*}
	In the same manner as \figref{shuffle-product}, a 3-way shuffle of $\mathbf{x}$, $\mathbf{y}$, and $\mathbf{z}$ can be thought of as a walk from $(0,0,0)$ to $(p,q,r)$. Terms appearing in (\ref{eq:Poisson-id-n-even}) are walks with a single bracket (see \figref{shuffle-product-bracket}), which may be one of the two types in \figref{diagonals-for-Poisson-id}.
	
	Consider a walk with a diagonal of the first type. It appears in $[\mathbf{x},\mathbf{y}\mathbf{z}]$ and $[\mathbf{x},\mathbf{y}]\mathbf{z}$ with the same sign and thus disappears in (\ref{eq:Poisson-id-n-even}).
	
	Otherwise, the walk has a diagonal of the second type. It appears in $[\mathbf{x},\mathbf{y}\mathbf{z}]$ and $\mathbf{y}[\mathbf{x},\mathbf{z}]$, but there is a sign difference of $(-1)^{|\mathbf{y}||\mathbf{x}|}$ because of the interchanged shuffle order of $\mathbf{x}$ and $\mathbf{y}$. Hence it also vanishes in (\ref{eq:Poisson-id-n-even}).
	
	We also sketch a visual proof of the bracket's compatibility with the coproduct in \figref{poisson-coproduct}, again in the case that $n$ is even. Let $\mathbf{x}=[x_1|x_2|x_3|x_4|x_5]$ and $\mathbf{y}=[y_1|y_2|y_3|y_4|y_5]$. In the figure, the term $\big[[x_1,y_1]|y_2|x_2|y_3\big] \otimes [y_4|x_3|x_4|y_5|x_5]$ is illustrated. On the left, it is depicted as it appears in $\Delta([\mathbf{x},\mathbf{y}])$; on the right, as it appears in $[\Delta(\mathbf{x}),\Delta(\mathbf{y})]$ as part of $\big[[x_1|x_2],[y_1|y_2|y_3]\big]\otimes [x_3|x_4|x_5][y_4|y_5]$. The difference in signs is
	\[
		(-1)^{\big|[x_3|x_4|x_5]\big|\cdot\big|[y_1|y_2|y_3]\big|},
	\]
	which is precisely the sign appearing in (\ref{eq:Poisson-tensor-product}).
	\begin{figure}[h]
		\centering
		\includegraphics[width=0.7\linewidth]{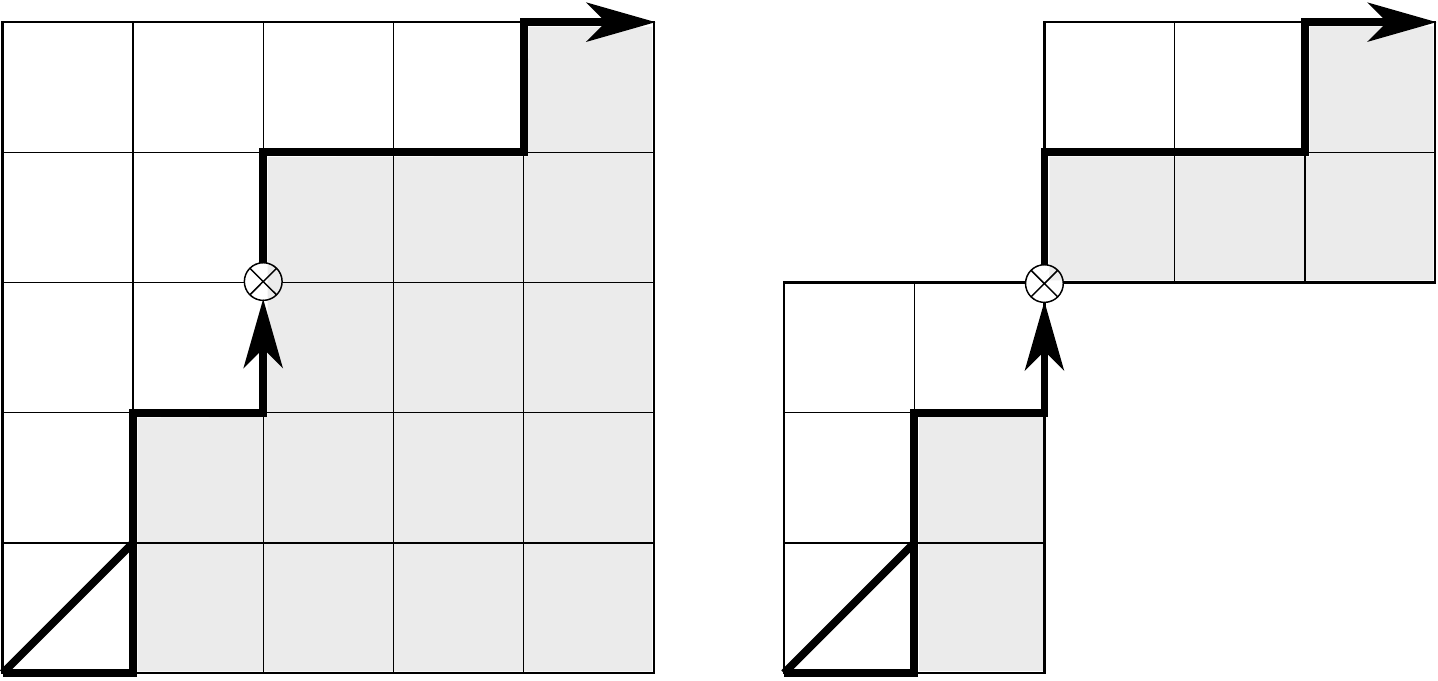}
		\caption{\label{fig:poisson-coproduct}A visual verification of $\Delta([\mathbf{x},\mathbf{y}]) = [\Delta(\mathbf{x}),\Delta(\mathbf{y})]$.}
	\end{figure}

	The above two properties of the bracket---the Poisson identity and compatibility with the coproduct---are ``automatic,'' in the sense that they do not rely on any properties of the bracket in $\B_{*,*}(\mathcal{A})$. They are mechanical consequences of how the product and coproduct in the bar construction are defined. However, checking the other requirements of a differential bigraded Poisson Hopf algebra will depend on properties of the bracket in $\B_{*,*}(\mathcal{A})$. We provide the following rough outline and we leave the details (as well as the case of $n$ odd) to the reader:
	\begin{itemize}
		\item The bracket in $\B_{*,*}(\mathcal{A})$ respects the external differential $\delta$ because the bracket in $\mathcal{A}$ satisfies the Poisson identity.
		\item The bracket in $\B_{*,*}(\mathcal{A})$ respects the internal differential $d$ because the bracket in $\mathcal{A}$ respects the differential $d$.
		\item The bracket in $\B_{*,*}(\mathcal{A})$ is antisymmetric because the bracket in $\mathcal{A}$ is.
		\item The bracket in $\B_{*,*}(\mathcal{A})$ satisfies the Jacobi identity because the bracket in $\mathcal{A}$ does.
	\end{itemize}
\end{proof}

\thmref{poisson-hopf-alg-bar-construction} is relevant to us because the $E^1_{*,*}$ page in the bar spectral sequence is precisely the bar construction on $H_*(\Omega^n X)$ (see \S\ref{subsec:bar-sseq}), where we think of $H_*(\Omega^n X)$ as a DGA with trivial differential. Our main result in the following section is that the bar spectral sequence is a spectral sequence of Poisson Hopf algebras, where the bracket on $E^1_{*,*}$ is as we described above.

\section{\label{sec:bracket-mainproof-section}The bracket in the bar spectral sequence}

\subsection{\label{sec:bracket-2-fold}Further structure in the total bar complex}

As mentioned in \S\ref{subsec:The-bar-construction}, $\B_{*,*}(\mathcal{A})$
has a multiplication when $\mathcal{A}$ is commutative. However,
$\mathcal{A}=C_{*}(\Omega G)$ is not commutative, so the multiplication
$\mathcal{A}\otimes\mathcal{A}\rightarrow\mathcal{A}$ is not an algebra
map. Nonetheless, the multiplication on $G$ induces a multiplication
on $B\Omega G$ which translates through (\ref{eq:topochainequiv})
to a multiplication on $\tot\B_{*,*}(\mathcal{A})$, as follows.

Looping the multiplication
$G\times G\rightarrow G$ gives a map $\Omega(G\times G)\rightarrow\Omega G$. Note that with the Moore loops model of $\Omega G$, the spaces $\Omega G\times\Omega G$ and $\Omega(G\times G)$ are generally different. We have a comparison map $\zeta\colon \Omega G \times \Omega G \to \Omega(G \times G)$ where $\zeta((\omega_{1},l_{1}),(\omega_{2},l_{2}))=((\omega_{1},\omega_{2}),\max\{l_{1},l_{2}\})$.
From this we obtain a product on $\Omega G$ which multiplies two loops pointwise and takes the longer of the two lengths:
\[
M_{0}\colon\Omega G\times\Omega G\xrightarrow{\zeta}\Omega(G\times G)\rightarrow\Omega G.
\]
This product is unital and associative; moreover it is homotopic to the concatenation product on $\Omega G$. We would like to deloop it to obtain a product on $B\Omega G$.

Sugawara \cite{sugawara} shows that if $Y_1$ and $Y_2$ are associative H-spaces, then a map $M_{0}\colon Y_{1}\rightarrow Y_{2}$
determines a map $BY_1 \rightarrow BY_2$ if for $n\geq1$ there exist
homotopies
\[
M_{n}\colon\underset{n\text{ copies of }I}{\underbrace{Y_{1}\times I\times Y_{1}\times I\times\cdots\times I\times Y_{1}}}\rightarrow Y_{2}
\]
satisfying
\[
M_{n}(y_{1},t_{1},\ldots,t_{n},y_{n+1})=\begin{cases}
M_{n-1}(y_{1},t_{1},\ldots,t_{i-1},y_{i}\times_{1}y_{i+1},t_{i+1},\ldots,t_{n},y_{n+1}) & \text{if }t_{i}=0,\\
M_{i-1}(y_{1},t_{1},\ldots,t_{i-1},y_{i})\times_{2}M_{n-i}(y_{i+1},t_{i+1},\ldots,t_{n},y_{n+1}) & \text{if }t_{i}=1,
\end{cases}
\]
where $\times_{1}$ and $\times_{2}$ denote the multiplications on
$Y_{1}$ and $Y_{2}$ respectively.

Clark \cite{homocomm} shows that $\zeta$ satisfies the above conditions,
so our $M_{0}$ can be delooped. We give a summary below.

$\Omega G$ has an outer multiplication
which concatenates two loops and an inner
multiplication $M_0$ which multiplies two loops pointwise. If
$l(a)$ denotes the length of $a$ and $c(t)$ denotes the constant
path of length $t$ at the basepoint of $G$, then explicitly $M_{1}$
is defined as
\begin{align*}
M_{1}((a_{1},b_{1}),t,(a_{2},b_{2})) & = M_0(a_{1}c\left(t(\max\{l(a_{1}),l(b_{1})\}-l(a_{1}))\right)a_{2},\\
 & \hspace*{1em} b_{1}c\left(t(\max\{l(a_{1}),l(b_{1})\}-l(b_{1}))\right)b_{2}).
\end{align*}
This is illustrated in \figref{M1-visual}, where the dotted segments
represent constant paths at the basepoint, and the total upper and
lower loops are multiplied pointwise.

\begin{figure}[h]
\includegraphics[width=5cm]{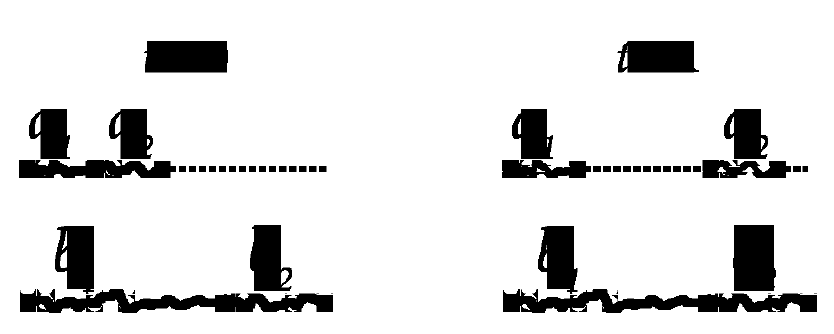}

\caption{\label{fig:M1-visual}The homotopy $M_{1}$ on $((a_{1},b_{1}),t,(a_{2},b_{2}))$
at $t=0,1$.}

\end{figure}

The reader is referred to Proposition 1.6 in \cite{homocomm} for
inductive definitions of higher $M_{n}$, but we will only need $M_{1}$
for now. These $M_{n}$ are used to construct the delooped map $B(\Omega G\times\Omega G)\rightarrow B\Omega G$.

We describe the corresponding map $\tot B_{*,*}(\mathcal{A}\otimes\mathcal{A})\rightarrow\tot B_{*,*}(\mathcal{A})$.
At the chain level, $M_{n}$ induces
\[
	\underset{n \text{ copies of }C_*(I)}{\underbrace{(\mathcal{A} \otimes \mathcal{A}) \otimes C_*(I) \otimes (\mathcal{A} \otimes \mathcal{A}) \otimes C_*(I) \otimes \cdots \otimes C_*(I) \otimes (\mathcal{A} \otimes \mathcal{A})}} \to \mathcal{A}.
\]
By taking the identity map $(\Delta^1 \to I) \in C_1(I)$ in each occurrence of $C_*(I)$, we get a map
\[
h_{n}\colon(\mathcal{A}\otimes\mathcal{A})^{\otimes(n+1)}\rightarrow\mathcal{A}
\]
which shifts degree by $n$. These assemble
into the map
\begin{equation*}
[x_{1}\otimes y_{1}|\cdots|x_{n}\otimes y_{n}] \mapsto\sum \Big[ h_{i_{1}-1}(x_{1}\otimes y_{1}|\cdots|x_{i_{1}}\otimes y_{i_{1}}) \Big|\cdots\Big|h_{i_{k}-1}(x_{n-i_{k}+1}\otimes y_{n-i_{k}+1}|\cdots|x_{n}\otimes y_{n})\Big]
\end{equation*}
where the sum is taken over all $k$-tuples (as $k$ ranges from $1$ to $n$) of positive integers $i_1,\ldots,i_k$ with $i_1 + \cdots + i_k = n$. Composing the above map with the shuffle product gives
\begin{align}
\B_{p,q}(\mathcal{A})\otimes\B_{s,t}(\mathcal{A})\xrightarrow{\text{EZ}}\B_{p+s,q+t}(\mathcal{A}\otimes\mathcal{A}) & \rightarrow\bigoplus_{\substack{m+n=p+q+s+t\\
m\leq p+s
}
}\B_{m,n}(\mathcal{A})\label{eq:generalbarproduct}\\
 & =F_{p+s}(\tot\B_{*,*}(\mathcal{A}))_{p+q+s+t}\nonumber 
\end{align}
which extends to a multiplication
on $\tot\B_{*,*}(\mathcal{A})$.

\subsection{The case $n=2$}

Browder's bracket on $H_{*}(\Omega X)$ has degree 0; it is just the
commutator $[x,y]=xy-(-1)^{\left|x\right|\left|y\right|}yx$. We consider the commutator on the bar complex and show that it induces a bracket on the spectral sequence which, by (\ref{eq:topologyiso}), converges to the commutator on $H_{*}(\Omega X)$.
The portion of (\ref{eq:generalbarproduct}) landing in $\B_{p+s,q+t}(\mathcal{A})$
is the shuffle product, which is commutative. Hence that portion vanishes
in the commutator, which is therefore a map with a $-1$ shift in
filtration degree.
\begin{prop}
\label{prop:bracket-2-fold-filtration-preserving}The bar spectral
sequence
\[
E_{*,*}^{2}\cong\Tor_{*,*}^{H_{*}(\Omega^{2}X)}(k,k)\Rightarrow\Tor_{*}^{C_{*}(\Omega^{2}X)}(k,k)
\]
is a spectral sequence of Poisson Hopf algebras, with bracket of bidegree $(-1,1)$ satisfying
\begin{equation}
	d^r[x,y] = [d^rx,y] + (-1)^{|x|}[x,d^r y].\label{eq:2-bracket-Leibniz-ss}
\end{equation}
\end{prop}
\begin{proof}
Because the multiplication respects the differential $D$ on the total
complex, the commutator does also:
\begin{align}
D[x,y] & =[Dx,y]+(-1)^{\left|x\right|}[x,Dy].\label{eq:bracket0Leibniz}
\end{align}
Abbreviate $\tot\B_{*,*}(\mathcal{A})$ as $J$. The commutator is a
map $F_{p}J\otimes F_{s}J\rightarrow F_{p+s-1}J$. We check that it
induces a bracket on the spectral sequence of the form $E_{p,q}^{*}\otimes E_{s,t}^{*}\rightarrow E_{p+s-1,q+t+1}^{*}$. The method is basically the same as in \cite[Theorem~2.14]{sseq user's guide}.
Let
\begin{align*}
	Z_{p,q}^{r} &= F_pJ_{p+q} \cap D^{-1}(F_{p-r}J_{p+q-1}),\\
	B_{p,q}^r &= F_p J_{p+q} \cap D(F_{p+r}J_{p+q+1}).
\end{align*}
Then the proof of \cite[Theorem~2.3]{sseq user's guide} identifies
\[
	E_{p,q}^r = Z_{p,q}^r / (Z_{p+1,q-1}^{r+1} + B_{p,q}^{r+1}).
\]
Take elements
\begin{align*}
x & \in Z_{p,q}^{r}=F_{p}J_{p+q}\cap D^{-1}(F_{p-r}J_{p+q-1})\\
y & \in Z_{s,t}^{r}=F_{s}J_{s+t}\cap D^{-1}(F_{s-r}J_{s+t-1})
\end{align*}
representing classes in $E_{p,q}^{r}$ and $E_{s,t}^{r}$, respectively.
By (\ref{eq:bracket0Leibniz}) we have
\[
[x,y]\in F_{p+s-1}J_{p+q+s+t}\cap D^{-1}(F_{p+s-r-1}J_{p+q+s+t-1})=Z_{p+s-1,q+t+1}^{r}.
\]
We leave it to the reader to check that bracketing an element of $Z_{p+1,q-1}^{r+1} + B_{p,q}^{r+1}$ with $y$ gives an element of $Z_{p+s,q+t}^{r+1} + B_{p+s-1,q+t+1}^{r+1}$, so we indeed get a map of the form $E_{p,q}^{*}\otimes E_{s,t}^{*}\rightarrow E_{p+s-1,q+t+1}^{*}$
as desired.

The bracket on each page of the spectral sequence
is induced by the one on $J$, so they are always compatible. In
particular, on the $E_{*,*}^{\infty}$ page, the bracket is the one
induced on $\gr_{*}H_{*}(J)$. The various requirements for a Poisson Hopf algebra will follow from \thmref{2-fold-main} and \thmref{poisson-hopf-alg-bar-construction}. Finally, (\ref{eq:2-bracket-Leibniz-ss}) follows from (\ref{eq:bracket0Leibniz}).
\end{proof}
Now that we know the bracket is well-behaved in the spectral sequence,
we explicitly compute it on the $E_{*,*}^{1}$ page to see how it
relates to the degree 1 bracket on $H_{*}(\Omega^{2}X)$.
\begin{thm}
\label{thm:2-fold-main}For $x,y\in H_{*}(\Omega^{2}X)$, let $[x,y]$
denote their degree 1 bracket. In the bar spectral
sequence converging to $H_{*}(\Omega X)$, the $E_{*,*}^{1}$ page is the bar construction on $H_*(\Omega^2 X)$, where the bracket is given by \thmref{poisson-hopf-alg-bar-construction}.
\end{thm}

\begin{proof}
Let $\mathbf{x}=[x_{1}|\cdots|x_{p}],\mathbf{y}=[y_{1}|\cdots|y_{q}]\in E_{*,*}^{0}$.
In the spectral sequence, only the portion of the bracket lying in
filtration $F_{m+n-1}$ is visible. Consider the terms contributed
by $\mathbf{xy}$ in the commutator\textemdash by (\ref{eq:generalbarproduct}),
these are shuffles with one $h_{1}$ operation thrown in. They have
one of four forms:
\begin{enumerate}
\item $[\cdots|h_{1}((x\otimes1)\otimes(x'\otimes1))|\cdots]$
\item $[\cdots|h_{1}((x\otimes1)\otimes(1\otimes y))|\cdots]$
\item $[\cdots|h_{1}((1\otimes y)\otimes(x\otimes1))|\cdots]$
\item $[\cdots|h_{1}((1\otimes y)\otimes(1\otimes y'))|\cdots]$
\end{enumerate}
where $x,x'$ are arbitrary $x_{i}$ and likewise for $y,y'$. However,
$h_{1}((x\otimes1)\otimes(x'\otimes1))$ and $h_{1}((1\otimes y)\otimes(1\otimes y'))$
are degenerate, and thus terms of the first and fourth types vanish.
\begin{figure}
\includegraphics[width=7cm]{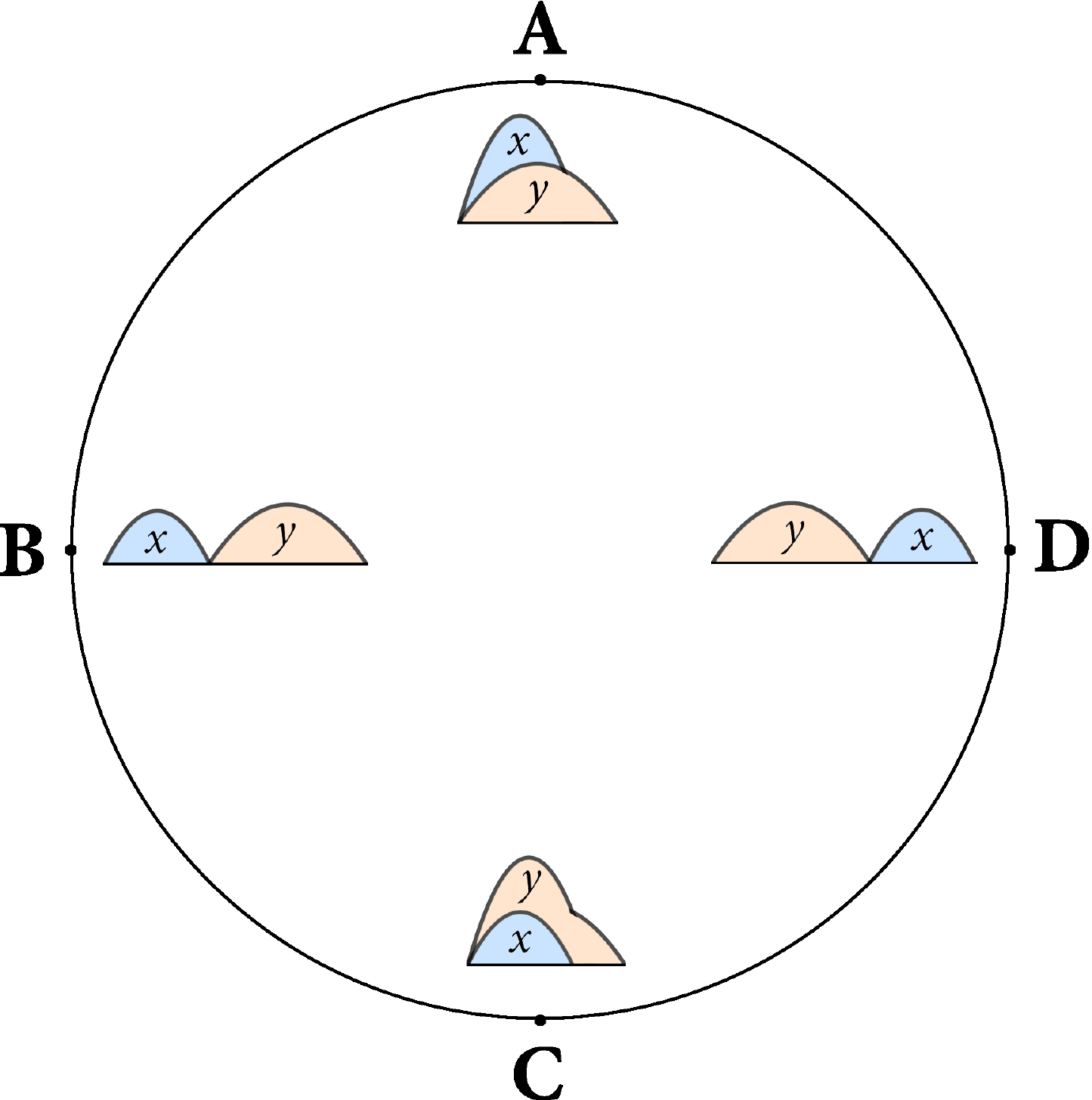}\caption{\label{fig:phi-for-degree-2}The map $\phi\colon\Omega^{2}X\times S^{1}\times\Omega^{2}X\rightarrow\Omega^{2}X$,
illustrated for $(x,t,y)$ as $t\in S^{1}$ varies. $x$ and $y$
are visualized as ``mounds'' of points in $X$, where their outlines
are mapped to the basepoint. The outer concatenation on $\Omega^{2}X$
is depicted horizontally and the inner (pointwise) concatenation is
depicted vertically.}
\end{figure}

Fix $x\in\{x_{i}\}$ and $y\in\{y_{i}\}$, as well as the shuffle
hidden by the ``$\cdots$'' in the bar expression, and consider
the term of the second form $[\cdots|h_{1}((x\otimes1)\otimes(1\otimes y))|\cdots]$.
Without loss of generality, assume that this shuffle has positive
sign in $\mathbf{xy}$. In the following, we refer to \figref{phi-for-degree-2}.
Let $\mathrm{in}_{\mathbf{AB}}\colon I\rightarrow S^{1}$ cover the
quarter-circle sending $0\mapsto\mathbf{A}$ and $1\mapsto\mathbf{B}$.
We then have the commutative diagram
\[
\begin{tikzcd}
& \Omega^2 X \times * \times I \times * \times \Omega^2 X \arrow[dl, hook'] \arrow[dd] \arrow[dr, hook, swap, "\mathrm{in}_\mathbf{AB}"]\\
\Omega^2 X \times \Omega^2 X \times I \times \Omega^2 X \times \Omega^2 X \arrow[dr, "M_1"] && \Omega^2 X \times S^1 \times \Omega^2 X \arrow[dl, swap, "\phi"]\\
& \Omega^2 X
\end{tikzcd}
\]
If we let $i_{\mathbf{AB}}\in C_{1}(S^{1})$ denote the 1-chain given
by $\mathrm{in}_{\mathbf{AB}}$, it follows that
\[
\phi_{*}(x\otimes i_{\mathbf{AB}}\otimes y)=h_{1}((x\otimes1)\otimes(1\otimes y)).
\]
Now consider $[\cdots|h_{1}((1\otimes y)\otimes(x\otimes1))|\cdots]$,
which has sign $(-1)^{(\left|x\right|+1)(\left|y\right|+1)}$ in $\mathbf{xy}$.
One similarly finds that
\begin{align*}
h_{1}((1\otimes y)\otimes(x\otimes1)) & =(-1)^{\left|x\right|\left|y\right|+\left|x\right|+\left|y\right|}\phi_{*}(x\otimes i_{\mathbf{AD}}\otimes y)\\
 & =(-1)^{(\left|x\right|+1)(\left|y\right|+1)}\phi_{*}(x\otimes-i_{\mathbf{AD}}\otimes y)
\end{align*}
where the sign $(-1)^{\left|x\right|\left|y\right|+\left|x\right|+\left|y\right|}$
comes from interchanging $x$ and $y$ across a 1-chain. Hence this
contributes the term $[\cdots|\phi_{*}(x\otimes-i_{\mathbf{AD}}\otimes y)|\cdots]$.

Likewise, we consider the terms that arise in $-(-1)^{\left|\mathbf{x}\right|\left|\mathbf{y}\right|}\mathbf{xy}$.
One is $[\cdots|h_{1}((1\otimes x)\otimes(y\otimes1))|\cdots]$, which
has sign $(-1)^{\left|\mathbf{x}\right|\left|\mathbf{y}\right|}$
in $\mathbf{yx}$ and thus negative sign in the commutator. But
\[
h_{1}((1\otimes x)\otimes(y\otimes1))=\phi_{*}(x\otimes i_{\mathbf{CB}}\otimes y)
\]
so it ultimately contributes the term $[\cdots|\phi_{*}(x\otimes-i_{\mathbf{CB}}\otimes y)|\cdots]$.
Lastly there is the term $[\cdots|h_{1}((y\otimes1)\otimes(1\otimes x))|\cdots]$
with sign $-(-1)^{(\left|x\right|+1)(\left|y\right|+1)}$, and
\[
h_{1}((y\otimes1)\otimes(1\otimes x))=(-1)^{\left|x\right|\left|y\right|+\left|x\right|+\left|y\right|}\phi_{*}(x\otimes i_{\mathbf{CD}}\otimes y),
\]
so it results in the term $[\cdots|\phi_{*}(x\otimes i_{\mathbf{CD}}\otimes y)|\cdots]$.

The sum $\gamma=i_{\mathbf{AB}}-i_{\mathbf{CB}}+i_{\mathbf{CD}}-i_{\mathbf{AD}}$
represents a generator of $H_{1}(S^{1})$, and altogether the set
of four terms considered above in the commutator combine to give $[\cdots|\phi_{*}(x\otimes\gamma\otimes y)|\cdots]$.
When we pass to the $E_{*,*}^{1}$ page, this becomes $[\cdots|[x,y]|\cdots]$
where $[x,y]$ is the bracket on $H_{*}(\Omega^{2}X)$ as defined
in \S\ref{subsec:Browder-bracket}.
\end{proof}

\subsection{The case $n\geq 3$}

We now turn to the case $n\geq 3$, with the aim of relating the degree $n-1$ bracket on $H_{*}(\Omega^n X)$ to the degree
$n-2$ bracket on $H_{*}(\Omega^{n-1}X)$ through the bar spectral sequence. Our treatment
will not be as explicit as in the $n=2$ case.

We proceed similarly as in \S\ref{sec:bracket-2-fold}, except now
we begin with $\phi_{n-1}\colon \Omega^{n-1} X\times S^{n-2}\times \Omega^{n-1} X\rightarrow \Omega^{n-1} X$.
Looping gives a map
\[
M_{0}\colon \Omega^{n} X\times\Omega S^{n-2}\times \Omega^{n} X\xrightarrow{\zeta}\Omega(\Omega^{n-1} X\times S^{n-2}\times \Omega^{n-1} X)\rightarrow \Omega^{n} X
\]
which operates ``pointwise'' on loops, and there are homotopies
\[
M_{n}\colon\underbrace{(\Omega^n X\times\Omega S^{n-2}\times\Omega^n  X)\times I\times\cdots\times I\times(\Omega^n X\times\Omega S^{n-2}\times \Omega^n X)}_{n\text{ copies of }I}\rightarrow \Omega^n X
\]
which are used to deloop $M_{0}$. At the chain level, we take the identity map $(\Delta^1 \to I) \in C_1(I)$ in each occurrence of $C_*(I)$ and obtain a family of chain maps
\[
h_{n}\colon(C_{*}(\Omega^n X)\otimes\Omega S^{n-2}\otimes C_{*}(\Omega^n X))^{\otimes(n+1)}\rightarrow C_{*}(\Omega^n X)
\]
where $h_n$ shifts degree by $n$.

In turn, we obtain a map
\begin{align*}
\B_{*,*}(C_{*}(\Omega^n X))\otimes\B_{*,*}(C_{*}(\Omega S^{n-2}))\otimes\B_{*,*}(C_{*}(\Omega^n X)) & \xrightarrow{\mathrm{EZ}}\B_{*,*}(C_{*}(\Omega^n X)\otimes C_{*}(\Omega S^{n-2})\otimes C_{*}(\Omega^n X))\\
 & \rightarrow\tot\B_{*,*}(C_{*}(\Omega^n X))
\end{align*}
where the first map is the shuffle product on three terms and the
second is as described in (\ref{eq:generalbarproduct}). The bracket
on the total bar complex is given by fixing $[\xi]\in\B{}_{*,*}(C_{*}(\Omega S^{n-2}))$
in the above, where $\xi\in C_{n-3}(\Omega S^{n-2})$ represents the
generator in homology. If $\beta\in C_{n-3}(S^{n-3})$ represents
a generator of $H_{n-3}(S^{n-3})$ and $\eta\colon S^{n-3}\rightarrow\Omega S^{n-2}$
is the unit of the loop-suspension adjunction (see \figref{loop-suspension-unit}),
then we can take $\xi=\eta_{*}(\beta)$. Note that this only makes sense for $n\geq 3$, which is why we considered $n=2$ separately.
\begin{figure}
	\includegraphics[width=6cm]{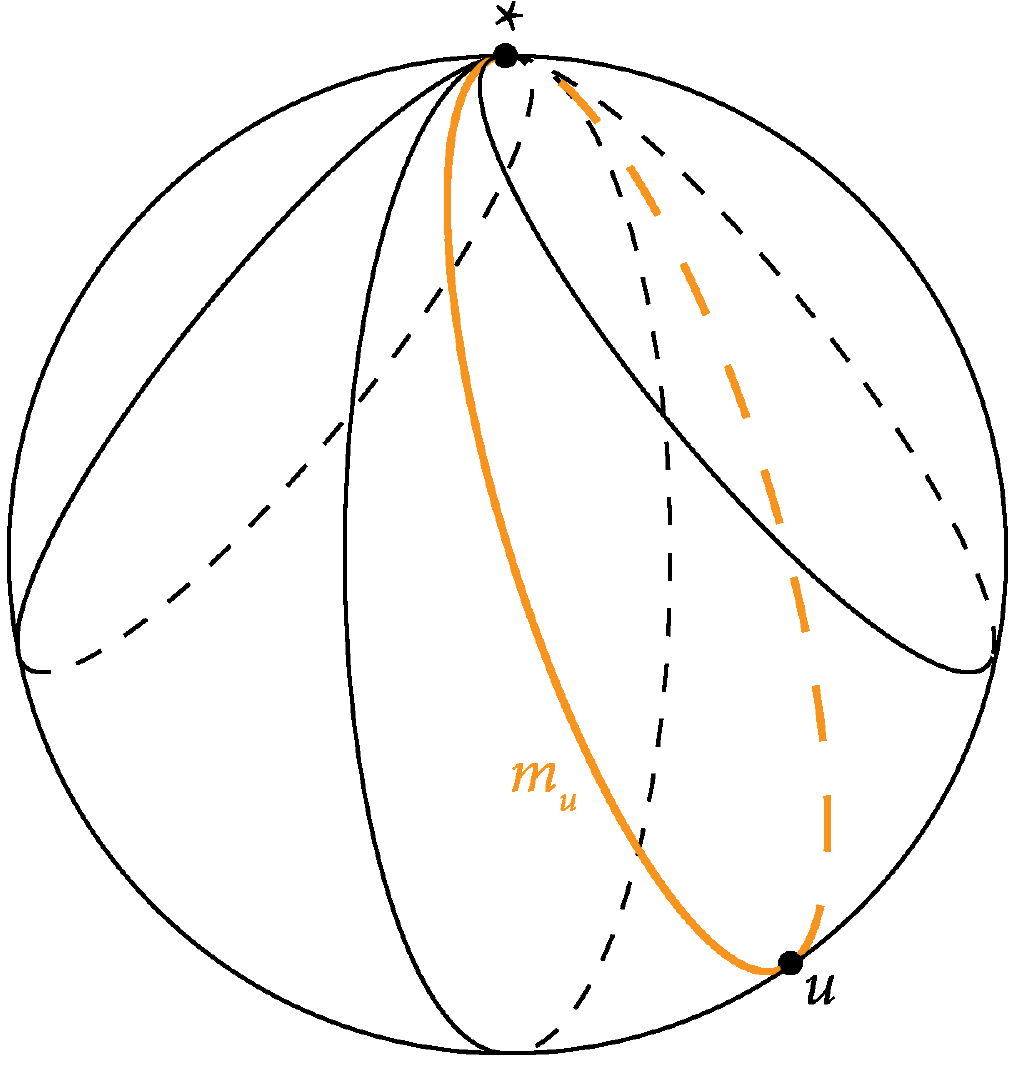}
	
	\caption{\label{fig:loop-suspension-unit}The map $\eta\colon S^{n-3}\rightarrow\Omega S^{n-2}$
		sending $u$ to the loop $m_{u}$, depicted for $n=4$.}
	
\end{figure}
\begin{prop}
Let $n\geq2$. The bar spectral sequence
\[
E_{*,*}^{2}\cong\Tor_{*,*}^{H_{*}(\Omega^{n}X)}(k,k)\Rightarrow\Tor_{*}^{C_{*}(\Omega^{n}X)}(k,k)
\]
is a spectral sequence of Poisson Hopf algebras, with bracket of bidegree $(-1,n-1)$ satisfying
\begin{equation}
d^r[x,y] = [d^rx,y] + (-1)^{n+|x|}[x,d^r y].\label{eq:n-bracket-Leibniz-ss}
\end{equation}
\end{prop}
\begin{proof}
We have already proved it for $n=2$ so let $n\geq3$. A priori, the
degree $n-2$ bracket on $J=\tot\B_{*}(C_{*}(\Omega^{n}X))$ maps $[-,-]\colon F_{p}J\otimes F_{q}J\rightarrow F_{p+q+1}J$.
The task is to show that the bracket actually lands in $F_{p+q-1}J$.

Each shuffle of $[x_{1}|\cdots|x_{p}]$, $[\xi]$, and $[y_{1}|\cdots|y_{q}]$
contains terms $x_{i}\otimes1\otimes1$, $1\otimes1\otimes y_{j}$,
and one instance of $1\otimes\xi\otimes1$. However, note that $h_{0}(1\otimes\xi\otimes1)$
gives a degenerate chain, thus vanishing in $C_{*}(X)$. Hence the
part of the bracket lying in $F_{p+q+1}$ vanishes.

The part of the bracket lying in $F_{p+q}$ is given by inserting
one $h_{1}$ into each shuffle. If the $h_{1}$ does not touch $1\otimes\xi\otimes1$,
then the result vanishes by the preceding. But even if the $h_{1}$
is inserted as $h_{1}((x\otimes1\otimes1)\otimes(1\otimes\xi\otimes1))$
for example, the result is \emph{still} degenerate, since the role
of $\Omega S^{n-2}$ is to control the multiplication on elements
of $\Omega^n X$. Thus the part of the bracket residing in $F_{p+q}$ also
vanishes.

Recall that this bracket comes from a chain map
\[
	J \otimes \tot\B_*(C_*(\Omega S^{n-2})) \otimes J \to J
\]
which sends $x\otimes [\xi] \otimes y$ to $[x,y]$ where $|[\xi]|=n-2$; the fact that this is a chain map implies (\ref{eq:n-bracket-Leibniz-ss}). The Poisson Hopf algebra axioms will follow from \ref{thm:generalmain} and \ref{thm:poisson-hopf-alg-bar-construction}. The remainder proceeds as in the proof of \propref{bracket-2-fold-filtration-preserving}.
\end{proof}
The terms of the bracket landing in $F_{p+q-1}$ are given by shuffles
with either one $h_{2}$ inserted or two $h_{1}$s inserted, but by
the same reasoning as in the above proof, shuffles with two $h_{1}$s
inserted will vanish. As such, an understanding of the homotopy $M_{2}$
is necessary. It is similar in spirit to $M_{1}$. In \figref{M2-visual},
we offer a visual description of $M_{2}$ which the reader should
compare to \figref{M1-visual}. An explicit definition can be constructed
from Proposition 1.6 in \cite{homocomm}.

\begin{figure}
\includegraphics[width=9cm]{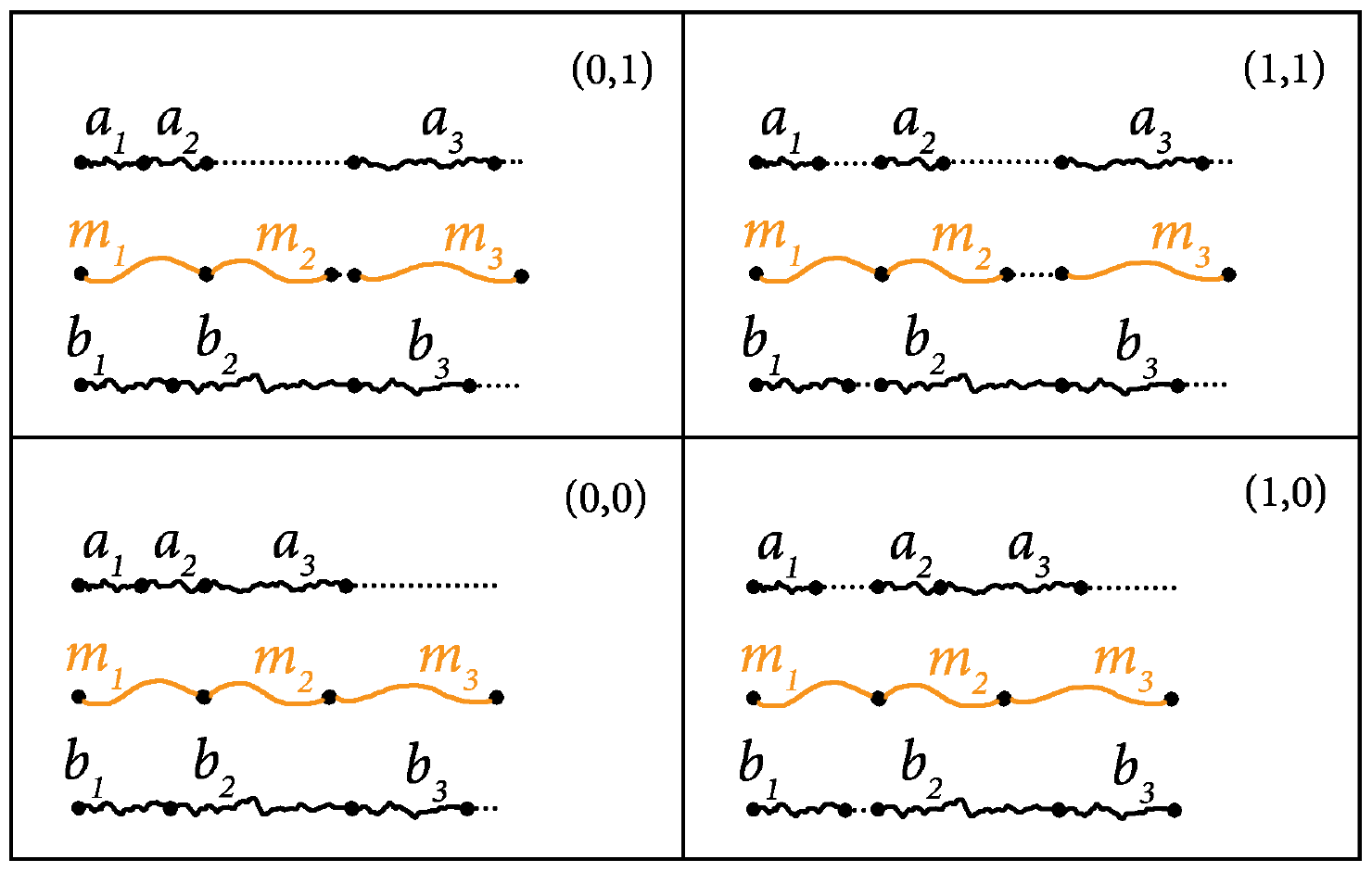}

\caption{\label{fig:M2-visual}The homotopy $M_{2}$ on $((a_{1},m_{1},b_{1}),t_{1},(a_{2},m_{2},b_{2}),t_{2},(a_{3},m_{3},b_{3}))$
as $(t_{1},t_{2})\in I^{2}$ varies.}
\end{figure}

Using $\eta\colon S^{n-3} \to \Omega S^{n-2}$, define
$\hat{M}_{n}$ to be the composite
\begin{align*}
\hat{M}_{n}\colon & (\Omega^n X\times S^{n-3}\times \Omega^n X)\times I\times\cdots\times I\times(\Omega^n X\times S^{n-3}\times \Omega^n X)\\
& \rightarrow(\Omega^n X\times\Omega S^{n-2}\times \Omega^n X)\times I\times\cdots\times I\times(\Omega^n X\times\Omega S^{n-2}\times \Omega^n X)\rightarrow \Omega^n X
\end{align*}
which at the chain level induces the degree $n$ map
\begin{equation}
\hat{h}_{n}\colon(C_{*}(\Omega^n X)\otimes C_{*}(S^{n-3})\otimes C_{*}(\Omega^n X))^{\otimes(n+1)}\rightarrow C_{*}(\Omega^n X).\label{eq:hn-hat}
\end{equation}

Now for our main result, we show that the bracket in the spectral sequence is as we described in \S\ref{sec:combinatorics-bracket-bar-construction}.
\begin{thm}
\label{thm:generalmain}Let $n\geq 2$. For $x,y\in H_{*}(\Omega^n X)$, let $[x,y]$ denote their degree
$n-1$ bracket. In the bar spectral
sequence converging to $H_{*}(\Omega^{n-1} X)$, the $E_{*,*}^{1}$ page is the bar construction on $H_*(\Omega^n X)$, where the bracket is given by \thmref{poisson-hopf-alg-bar-construction}.
\end{thm}
\begin{proof}
The sign $(-1)^{\sigma'(\varphi,i)}$ appears only when $n$ is odd,
so when $n=2$ this reduces to \thmref{2-fold-main}.

Let $n\geq3$. First we consider $E_{*,*}^{0}$. The $h_{2}$ inserted
into a shuffle must take $1\otimes\xi\otimes1$ as one of its arguments,
or the shuffle will vanish. If the other two arguments are both of
the form $x\otimes1\otimes1$, or both of the form $1\otimes1\otimes y$,
then the result is again degenerate and vanishes. In light of these
considersations, if we fix the surrounding shuffle denoted by ``$\cdots$''
below, the bracket appears in sets of 6 pieces (with $\beta$ and
$\hat{h}_{2}$ as in the discussion preceding (\ref{eq:hn-hat})):
\begin{enumerate}
\item $[\cdots|\hat{h}_{2}((1\otimes\beta\otimes1)\otimes(x\otimes1\otimes1)\otimes(1\otimes1\otimes y))|\cdots]$
\item $[\cdots|\hat{h}_{2}((x\otimes1\otimes1)\otimes(1\otimes\beta\otimes1)\otimes(1\otimes1\otimes y))|\cdots]$
\item $[\cdots|\hat{h}_{2}((x\otimes1\otimes1)\otimes(1\otimes1\otimes y)\otimes(1\otimes\beta\otimes1))|\cdots]$
\item $[\cdots|\hat{h}_{2}((1\otimes\beta\otimes1)\otimes(y\otimes1\otimes1)\otimes(1\otimes1\otimes x))|\cdots]$
\item $[\cdots|\hat{h}_{2}((y\otimes1\otimes1)\otimes(1\otimes\beta\otimes1)\otimes(1\otimes1\otimes x))|\cdots]$
\item $[\cdots|\hat{h}_{2}((y\otimes1\otimes1)\otimes(1\otimes1\otimes x)\otimes(1\otimes\beta\otimes1))|\cdots]$
\end{enumerate}
Each of the above $\hat{h}_{2}$ expressions can be induced from a
restriction of $\hat{M}_{2}$. For example, terms of the second form
are induced by the restriction of $\hat{M}_{2}$ to
\[
F_{2}\colon(\Omega^n X\times*\times*)\times I\times(*\times S^{n-3}\times*)\times I\times(*\times*\times \Omega^n X)\rightarrow \Omega^n X.
\]
Define the restrictions $F_{i}$ for $i=1,\ldots,6$ analogously.
These can be stitched together to be parts of a larger map $F\colon \Omega^n X\times S^{n-3}\times I^{2}\times \Omega^n X\rightarrow \Omega^n X$,
as illustrated in \figref{total-F}. The table beneath it shows how
$F$ acts at specific points. The loop $m_{u}\in\Omega S^{n-2}$ controls
the pointwise multiplication of loops $a,b\in \Omega^n X=\Omega (\Omega^{n-1} X)$. The dotted
segments denote constant paths at the basepoint.

\begin{figure}[h]
\includegraphics[width=14cm]{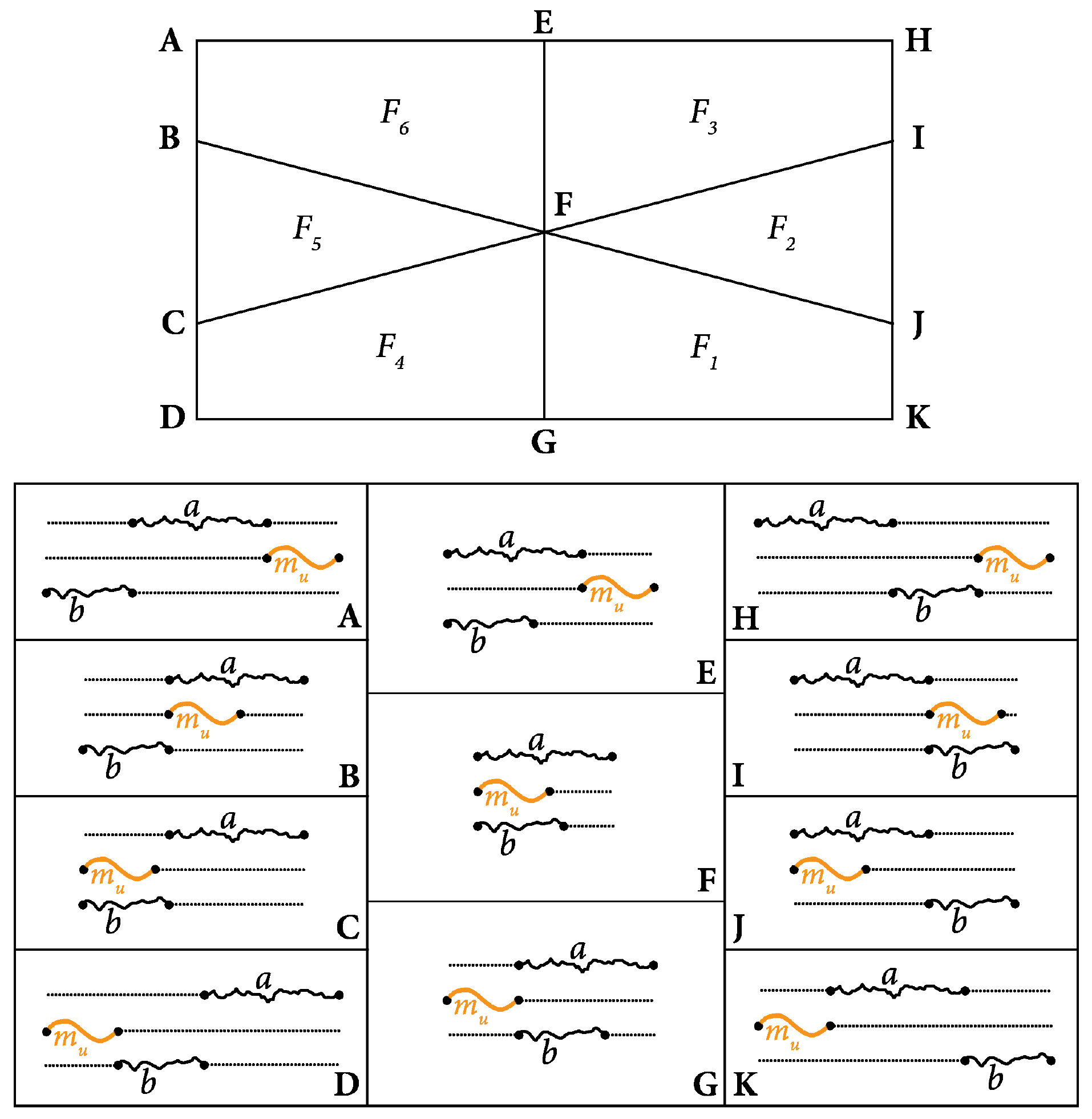}

\caption{\label{fig:total-F}The pieces $F_{i}$ assembled into a single map
$F\colon \Omega^n X\times S^{n-3}\times I^{2}\times \Omega^n X\rightarrow \Omega^n X$, with $a,b\in \Omega^n X$ (thought of as loops in $\Omega^{n-1}X$) and
$u\in S^{n-3}$. The map $u\protect\mapsto m_{u}$ is as in \figref{loop-suspension-unit}.}

\end{figure}

Note that where $m_{u}$ does not interact with $a$ or $b$, it produces
constant paths at the basepoint. For instance, evaluating $F$ at
the point labeled $\mathbf{A}$ gives: $b$, followed by $a$, followed
by a constant path (of length equal to that of $m_{u}$). By removing
all such extraneous constant paths, we obtain another map $\tilde{F}\colon \Omega^n  X\times S^{n-3}\times I^{2}\times \Omega^n  X\rightarrow \Omega^n  X$
which is homotopic to $F$. However, $\tilde{F}$ on the segment $\overline{\mathbf{AH}}$
is the same as on $\overline{\mathbf{DK}}$, and $\tilde{F}$ is constant
on $\overline{\mathbf{AD}}$ as well as on $\overline{\mathbf{HK}}$.
Making these identifications, the depicted rectangle $\mathbf{ADKH}$
becomes $S^{2}$. Moreover, the value of $\tilde{F}$ on $\overline{\mathbf{AH}}=\overline{\mathbf{DK}}$
is independent of $m_{u}$, and when $u=*\in S^{n-3}$, the value
of $\tilde{F}$ does not depend on vertical position in the rectangle
$\mathbf{ADKH}$ (i.e. the picture in \figref{total-F} can be flattened
to just the line $\overline{\mathbf{AH}}=\overline{\mathbf{DK}}$
for $u=*$). Thus $\tilde{F}$ factors through $\Omega^n  X\times S^{n-3}\wedge S^{2}\times \Omega^n  X\cong \Omega^n  X\times S^{n-1}\times \Omega^n  X$.

Now we have a map $\tilde{F}\colon \Omega^n  X\times S^{n-1}\times \Omega^n  X\rightarrow \Omega^n  X$,
and in homology (the $E_{*,*}^{1}$ page) the total of the six terms
listed at the beginning is $[\cdots|\tilde{F}_{*}(x,\gamma,y)|\cdots]$
for a generator $\gamma\in H_{n-1}(S^{n-1})$.

The map $\tilde{F}$ involves an unwanted ``twisting'' of $a$ and
$b$ according to $m_{u}$. This is remedied by the ``tting''
homotopy $H_{t}$ in \figref{untwist-homotopy}. Let
\[
L=\frac{L_{0}+l(m_{u})}{L_{h}+l(m_{u})}L_{h}.
\]
The homotopy multiplies $a$ and $b$ according to the bold path,
which coincides with $m_{u}$ on the interval $[Lt,L_{h}(1-t)+Lt]$
and is extended by constant paths on both sides.

\begin{figure}[h]
\includegraphics[width=10cm]{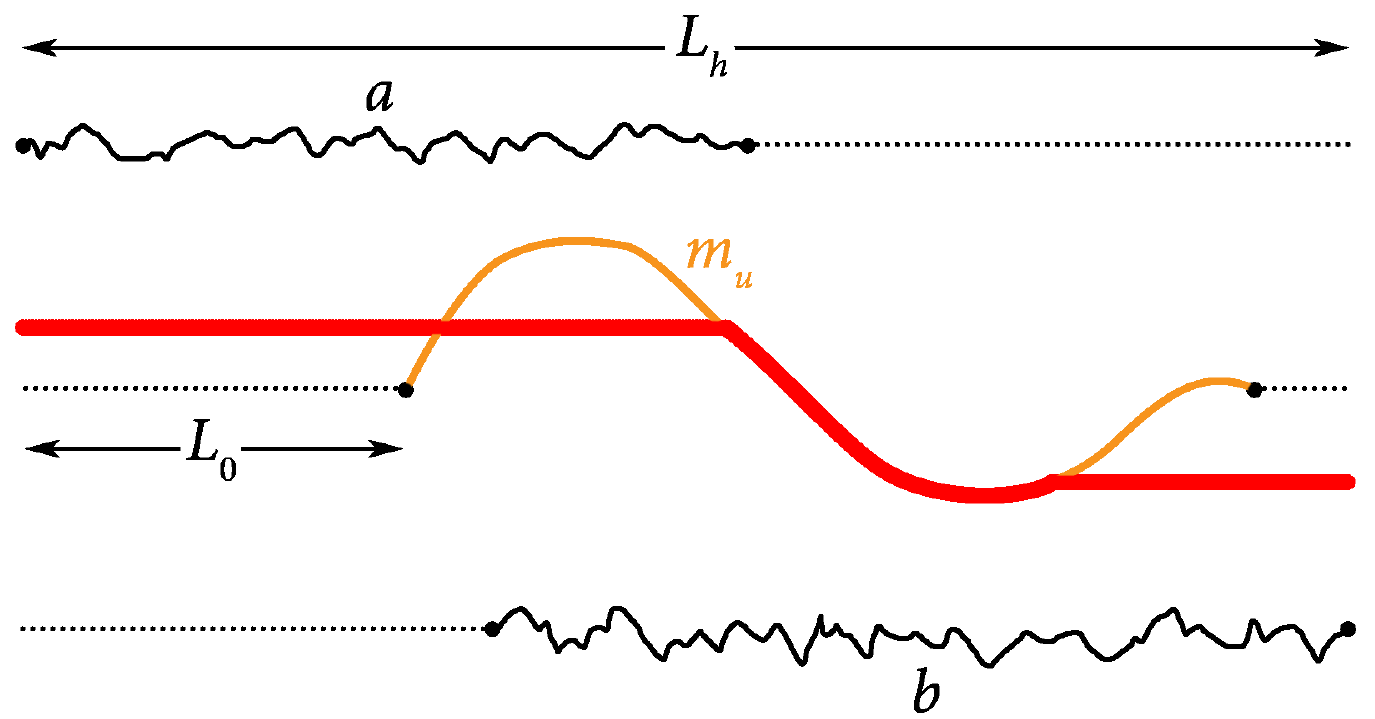}

\caption{\label{fig:untwist-homotopy}The homotopy $H_{t}\colon \Omega^n  X\times S^{n-1}\times \Omega^n  X\rightarrow \Omega^n  X$,
$t\in I$. (Although $L_{0}$ as drawn is positive, it can be in the
interval $[-l(m_{u}),L_{h}]$.) }
\end{figure}

At $t=1$ we obtain the desired $\phi\colon \Omega^n  X\times S^{n-1}\times \Omega^n  X\rightarrow \Omega^n  X$,
and $\tilde{F}_{*}(x,\gamma,y)=\phi_{*}(x,\gamma,y)=[x,y]$.

Lastly, the sign $(-1)^{\sigma'(\varphi,i)}$ as defined in \thmref{poisson-hopf-alg-bar-construction} results from terms $x_{j}\otimes1\otimes1$ being shuffled
after $1\otimes\xi\otimes1$ (which has total degree $n-2=n$ mod
2), and terms $1\otimes1\otimes y_{j}$ being shuffled before it\textemdash a
sign which is not accounted for in the shuffle of $[x_{1},\ldots,x_{p}]$
only with $[y_{1},\ldots,y_{q}]$.
\end{proof}

The bracket in the spectral sequence has a very simple description on $E_{1,*}^{1}$.
\begin{cor}
\label{cor:E1-bracket}If $x,y\in H_{*}(\Omega^n  X)$, then
\begin{equation}
[[x],[y]]=[[x,y]]\label{eq:column1-bracket}
\end{equation}
where $[x],[y]\in E_{1,*}^{1}$, $[[x],[y]]$ denotes the bracket
in the spectral sequence, and $[x,y]$ denotes the bracket in $H_{*}(\Omega^n  X)$.
\end{cor}
This corollary implies Theorem 2-1 in \cite{browder bracket}, as
the spectral sequence has an edge homomorphism $E_{1,*}^{1}\rightarrow H_{*}(\Omega^{n-1} X)$
which sends $[x]$ with $x\in H_{*}(\Omega^n  X)$ to the homology suspension
$\sigma x\in H_{*}(\Omega^{n-1} X)$.


\begin{thebibliography}{\textbf{1}}
\bibitem[\textbf{1}]{browder bracket}Browder, William. \emph{Homology
operations and loop spaces}. Illinois J. Math. 4 (1960), 347\textendash 357.

\bibitem[\textbf{2}]{homocomm}Clark, Allan. \emph{Homotopy commutativity
and the Moore spectral sequence}. Pacific J. Math. 15 (1965), 65-74.

\bibitem[\textbf{3}]{Homology-Iterated-Loop-Spaces}Cohen, Frederick
R.; Lada, Thomas J.; May, Peter J. \emph{The homology of iterated
loop spaces}. Berlin: Springer (1976).

\bibitem[\textbf{4}]{Eilenberg-MacLane53}Eilenberg, Samuel; Mac Lane, Saunders. \emph{On the groups $H(\Pi,n)$}. Annals of Mathematics, Second Series, 58 (1953), 55-106.

\bibitem[\textbf{5}]{DL in bar ss}Ligaard, Hans; Madsen, Ib. \emph{Homology
operations in the Eilenberg-Moore spectral sequence}. Math. Z. 143
(1975), 45\textendash 54.

\bibitem[\textbf{6}]{sseq user's guide}McCleary, John. \emph{A User's
Guide to Spectral Sequences}. Cambridge: Cambridge U, (2008).

\bibitem[\textbf{7}]{cartan-moore seminar}Moore, John C. \emph{Alg\`{e}bre
homologique et homologie des espaces classifiants}. S\'{e}minaire Henri
Cartan 12.1 (1959/60).

\bibitem[\textbf{8}]{Moore56}Moore, John C. \emph{Seminar Notes}, Princeton, 1956

\bibitem[\textbf{9}]{sugawara}Sugawara, Masahiro. \emph{On the homotopy-commutativity
of groups and loop spaces}. Memoirs of the College of Science, University
of Kyoto, Series A, Math. 33 (1960/61), 251-259.
\end{thebibliography}
\end{document}